 \documentclass[11pt]{amsart}
\usepackage{geometry}   
\usepackage{setspace}
\usepackage{graphicx}
\usepackage{amssymb, amsthm}
\usepackage{epstopdf}
\usepackage{pinlabel}
\usepackage{color}
\usepackage{hyperref}
\usepackage{tikz}
\usetikzlibrary{calc,decorations.markings}
 \newtheorem{theorem}{Theorem}[section]
 \newtheorem*{maintheorem}{Main Theorem}
    
      \newtheorem*{cor}{Corollary}
   \newtheorem{lemma}[theorem]{Lemma}
    \newtheorem{proposition}[theorem]{Proposition}  
       
      \newtheorem{question}[theorem]{Question} 
      
    \theoremstyle{definition}
\newtheorem{definition}[theorem]{Definition}
\newtheorem{remark}[theorem]{Remark}

\newtheorem{example}[theorem]{Example}

  \newcommand{\Z}{\ensuremath{{\mathbb{Z}}}}
\newcommand{\R}{\ensuremath{{\mathbb{R}}}}
  \newcommand{\N}{\ensuremath{{\mathbb{N}}}}

\definecolor{dred}{rgb}{.5,0,0} 
\definecolor{dgreen}{rgb}{0,.5,0} 
\definecolor{blue}{rgb}{0,0,0.5} 
\definecolor{black}{rgb}{0,0,0} 
\definecolor{vdgreen}{rgb}{0,.3,0} 
\definecolor{vdred}{rgb}{.3,0,0} 

\newcommand{\red}[1]{{\leavevmode\color{red}#1}}

\newcommand{\mbX}{\ensuremath{{\partial_* X}}}
\newcommand{\mbdX}{\ensuremath{{\partial_*^N X}}}
\newcommand{\mbY}{\ensuremath{{\partial_* Y}}}

\newcommand{\mbeY}{\ensuremath{{\partial_*^{N'} Y}}}
\newcommand{\mbh}{\ensuremath{{\partial_*h}}}

\newcommand{\diam}{\textrm{diam}}

\title{Quasi-Mobius Homeomorphisms of Morse boundaries}
\author{Ruth Charney, Matthew Cordes and Devin Murray}

\thanks {Charney was partially supported by NSF grant DMS-1607616.  Cordes was supported by a Zuckerman STEM Leadership Postdoctoral Fellowship.}

\begin{document}

\begin{abstract} The Morse boundary of a proper geodesic metric space is designed to encode hypberbolic-like behavior in the space.  A key property of this boundary is that a quasi-isometry between two such spaces induces a homeomorphism on their Morse boundaries.  In this paper we investigate when the converse holds.  We prove that for $X, Y$ proper, cocompact spaces, a homeomorphism between their Morse boundaries is induced by a quasi-isometry if and only if the homeomorphism is quasi-mobius and 2-stable.
\end{abstract}

\maketitle

\section{Introduction}

Boundaries of hyperbolic spaces have played a major role in the study of hyperbolic geometry and hyperbolic groups.  In particular, they provide a fundamental tool for studying the dynamics of isometries and rigidity properties of hyperbolic groups.

The effectiveness of this tool depends on a few key properties.  The first, is quasi-isometry invariance:  a quasi-isometry between two hyperbolic metric spaces induces a homeomorphism on their boundaries.  
In particular, this allows us to talk about the boundary of a hyperbolic group.  
Moreover, these homeomorphisms satisfy some particularly nice properties; they are quasi-mobius and quasi-conformal. Quasi-mobius is a condition that bounds the distortion of cross-ratios while quasi-conformal bounds the distortion of metric spheres.  These conditions have been studied in a variety of contexts by Otal, Pansu, Tukia, and Vaisala, \cite{O, Pan,T86,TV82, TV84, V} among others.  One of the most general theorems can be found in a 1996 paper of Paulin \cite{Pau} where he proves that if  $f : \partial X \to \partial Y$  is a homeomorphism between the boundaries of two proper, cocompact hyperbolic spaces, then the following are equivalent
\begin{enumerate}
\item $f$ is induced by a quasi-isometry $h : X \to Y$,
\item $f$ is quasi-mobius,
\item $f$ is quasi-conformal.
\end{enumerate}
We remark that Paulin's definition of quasi-conformal is different from the one used by Tukia and others. In this paper, we will focus on the quasi-mobius condition.

Boundaries can be defined for a variety of other spaces.  In particular, one can define a  boundary for any CAT(0) space in a similar manner to hyperbolic spaces.  Unfortunately, many of the nice properties of hyperbolic boundaries fail in this context.  First, quasi-isometries of CAT(0) spaces do not, in general, induce homeomorphisms on their boundaries. A well-known example of Croke and Kleiner \cite{CK} exhibits a group acting geometrically on two CAT(0) spaces with non-homeomorphic boundaries.  The missing property that leads to the failure of quasi-isometry invariance, is that in hyperbolic spaces, quasi-geodesics stay bounded distance from geodesics (with the bound depending only on the quasi-constants) while in CAT(0) spaces, this need not hold.  This property is known as the Morse property.  

In \cite{CS} the first author and H.~Sultan introduced a new type of boundary for CAT(0) spaces by restricting to only those geodesic rays satisfying the Morse property.  For CAT(0) spaces, the Morse property is equivalent to the contracting property and the authors originally called their boundary the ``contracting boundary".  Subsequently, their construction was generalized to arbitrary proper geodesic metric spaces by M.~Cordes \cite{Co} using the Morse property.  These boundaries have thus come to be known as Morse boundaries.  We denote the Morse boundary of $X$ by $\mbX$.  The key property of this boundary is quasi-isometry invariance; a quasi-isometry between two proper geodesic metric spaces induces a homeomorphism on their Morse boundaries \cite{CS, Co}.  Thus the Morse boundary is well-defined for \emph{any} finitely generated group (though it may be empty if the group has no Morse geodesics).  The Morse boundary is designed to behave like boundaries of hyperbolic groups and, hopefully, to have similar applications to more general groups. Evidence of this may be found in several papers including \cite{CS, Co, cordes-hume, cordes-durham, CaMa, Antolin:2016aa, Tran:2017aa, Mu}. For an overview of what is currently known about Morse boundaries, see Cordes' survey paper \cite{Co17}.
  
In the current paper, we will investigate the question of when a homeomorphism of Morse boundaries is induced by a quasi-isometry of the underlying groups.  We prove the following analogue of Paulin's theorem.

\begin{maintheorem} Let $X$ and $Y$ be proper, cocompact geodesic metric spaces and assume that $\mbX$ contains at least 3 points.  Then a homeomorphism $f : \mbX \to \mbY$ is induced by a quasi-isometry  
$h : X \to Y$ if and only if $f$ is 2-stable and quasi-mobius.  
\end{maintheorem}
We refer the reader to Section \ref{2-stable} for the definitions of quasi-mobius and 2-stable.

In particular, consider the case of a CAT(0) group.  In \cite{Mu}, building on work of Ballman and Buyalo \cite{BB}, the second author showed that if $G$ acts geometrically on a CAT(0) space $X$, then $\mbX$ contains at least 3 points if and only if $G$ is rank one and not virtually cyclic. Thus we have,
\begin{cor}  Let $G$ be a rank one CAT(0) group and $H$ any finitely generated group.  Then $H$ is quasi-isometric to $G$ if and only if there exists a homeomorphism $f : \partial_*G \to \partial_*H$ which is quasi-mobius and 2-stable.
\end{cor}

One might ask if there is also an equivalent quasi-conformality condition as in Paulin's theorem.  In general, however, the Morse boundary is neither metrizable nor compact, so it is not even clear what quasi-conformal should mean in this context.  However, a recent paper of Cashen and Mackey \cite{CaMa} introduces a metrizable topology on the Morse boundary which could potentially be used to define quasi-conformal.  
It would be interesting to know whether a full analogue of Paulin's theorem holds for this modified Morse boundary.  

An earlier version of this paper, dealing only with CAT(0) spaces, is available on the arXiv \cite{ChMu}.  The proofs in that setting are somewhat easier.
We have also recently learned that Sarah Mousley and Jacob Russel have proved an analogous result for Hierarchically Hyperbolic Groups (HHGs) \cite{MoRu}.

The first author would like to thank the Mathematical Sciences Research Institute in Berkeley and the Isaac Newton Institute for Mathematical Sciences in Cambridge for their support and hospitality during the writing of this paper.  Work at the Newton Institute was supported by EPSRC grant no EP/K032208/1.

\section{Preliminaries}\label{prelim}

\emph{We assume throughout the paper that $X$ is a proper geodesic metric space.}

\subsection{Morse triangles}  In this section we review some basic facts about Morse geodesics and define the Morse boundary.  The reader is referred to \cite{Co, BF, CS, Su} for more details.

\begin{definition}  A geodesic $\alpha$ in $X$ is \emph{Morse} if there exists a function $N: \R^+ \times \R^+ \to \R^+$ such that any $(\lambda, \epsilon)$-quasi-geodesic with endpoints on $\alpha$,  lies in the $N(\lambda, \epsilon)$-neighborhood of $\alpha$.  The function $N$ is called a \emph{Morse gauge} for $\alpha$ and we say that $\alpha$ is $N$-Morse.
 \end{definition}
 
 It is well known that if $X$ is hyperbolic, then there exists a Morse gauge $N$ such that every geodesic in $X$ is $N$-Morse.  At the other extreme, there are spaces $X$, such as the Euclidean plane, where no infinite geodesic is Morse.  In general, one has a mixture:  some infinite geodesics in $X$ are Morse and others not.  For example, if $X$ is the universal cover of the wedge of two tori, $T^2 \vee T^2$, then a geodesic $\alpha$ in $X$ is Morse if and only if there is a uniform bound on the amount of time $\alpha$ spends in any given flat.  That bound determines the Morse guage, the longer $\alpha$ spends in a flat, the larger the Morse guage $N$. 
 
Morse geodesics in a proper geodesic metric space $X$ behave much like geodesics in a hyperbolic metric space,
as indicated in the next few lemmas.  
 
\begin{lemma}[Equivalent Geodesics] \label{lem:equiv geodesics} Let $\alpha$ and $\beta$ be bi-infinite geodesics based at a point $p$.  Suppose  $\alpha$ is $N$-Morse and  the Hausdorff distance between $\alpha$ and $\beta$ is finite.  Then there exists a constant $C_N$ and a Morse gauge $N'$, depending only on $N$, such that $\beta$ is $N'$-Morse and the Hausdorff distance between $\alpha$ and $\beta$ is at most $C_N$.  
\end{lemma}

\begin{proof}  This follows from the proof of Proposition 2.4 in \cite{Co}. 
\end{proof}

Recall that a triangle is \emph{$\delta$-slim} if each side is contained in the $\delta$-neighborhood of the other two sides.

\begin{lemma}[Slim Triangles] \label{lem:slim triangles}  
	Let $T(a,b,c)$ be a triangle with vertices $a,b,c \in X \cup \mbX$ and suppose that all three edges of the triangle are $N$-Morse.  Then there exists a constant $\delta_N$ depending only on $N$ so that $T(a,b,c)$  is $\delta_N$-slim. 
\end{lemma}

\begin{proof}  For triangles with vertices in $X$, this follows from Lemma 2.3 of \cite{Co}.  That is, there exists $D_N$, depending only on $N$, such that interior triangles with $N$-Morse sides are $D_N$-slim.

Following Kent and Leininger's argument in Theorem 4.4 of \cite{Kent-Leininger}, we can extend this to ideal triangles by truncating the triangle at each ideal vertex.  For simplicity we will describe the argument in the case where all three vertices are on the boundary.  The same idea works when only some of the vertices are on the boundary.

It follows from the proof of Proposition 2.4 in \cite{Co} that for two $N$-Morse geodesic rays $\alpha$ and $\beta$ asymptotic to the same point on the boundary, there exists $L$, depending only on $N$, such that $\alpha$ and $\beta$ eventually lie within $L$ of each other, that is, for some $T_1,T_2$, the Hausdorff
distance from $\alpha([T_1, \infty))$ to $\beta([T_2, \infty))$ is at most $L$.  

Thus at each point $\{a,b,c\}$, we can truncate the triangle $T(a,b,c)$ by adding an edge of length at most $L$ to get a geodesic hexagon $P$ in $X$.   Then choose 3 vertices $a',b',c'$ of $P$  that form a central triangle $R(a',b',c')$ and three other peripheral triangles which share two sides with $P$ (see Figure \ref{fig:slim}). Each of these peripheral triangles contains one edge that is a segment of a side of $T$, one edge in $R$, and a third edge of length at most $L$.  By an elementary argument, subsegments of $N$-Morse geodesics are $N'$-Morse where $N'$ depends only on $N$, so the first of these edges is $N'$-Morse.  The union of the other two edges is a $(1,L)$-quasi-geodesic and is hence contained in the $N'(1,L)$-neighborhood of the first edge, in fact, by Lemma 2.1 of \cite{Co} they have Hausdorff distance $2N'(1,L)$. By Lemma 2.5 of \cite{CS} this implies that the sides of $R$ are $N''$-Morse where $N''$ depends only on $N',L$. 

Since $R$ has vertices in $X$, it is $D_{N''}$-slim,
and the sides of $P$ that are contained in $T$, lie in the $2N'(1,L)$-neighborhood of $R$.
Thus,taking $\delta_N = D_{N''} +4N'(1,L) + L$, we conclude that $T$ is $\delta_N$-slim.
\end{proof}

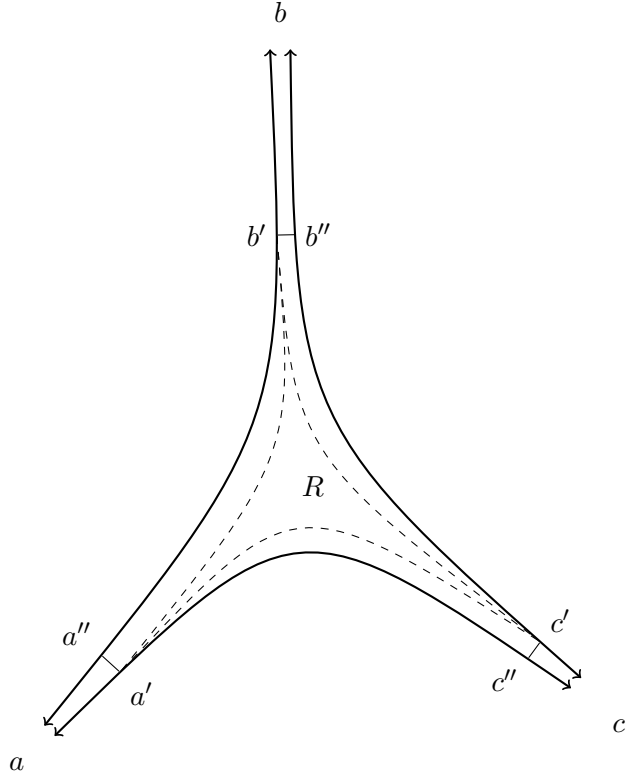
\begin{figure}
\begin{tikzpicture}

\node at (0,0) (a){};
\node at (3,9) (b){};
\node at (7,0.5) (c){};
\node at (3.2,4) (ctl1){};
\node at (3.2,3.6) (ctl2){};

\node at ($(a)-(0.5,0.5)$) {$a$};
\node at ($(b)+(0,0.5)$) {$b$};
\node at ($(c)+(0.5,-0.5)$) {$c$};

\draw[<->,thick] (a.west) ..controls ($(ctl1)-(0.1,0)$) .. ($(b.west)$) node [pos=0.8,scale=0.1](abm){} node [pos=0.085,scale=0.1](abm2){};
\draw[<->,thick] (b.east) ..controls (ctl1) .. (c.north) node [pos=0.95,scale=0.1](bcm){} node [pos=0.2,scale=0.1](bcm2){};
\draw[<->,thick] (c.west) ..controls ($(ctl1)-(0,1)$) .. (a.south) node [pos=0.9,scale=0.1](acm){} node [pos=0.055,scale=0.1](acm2){};

\draw[thin, dashed] (abm) ..controls ($(ctl2)-(0,0)$).. (bcm) node[pos=0.55,anchor=north east,black]{$R$} ;
\draw[thin, dashed] (acm)  ..controls ($(ctl2)-(0,0.4)$).. (bcm);
\draw[thin, dashed] (abm)  ..controls ($(ctl2)-(0,0)$)..  (acm);

\draw[thin] (abm) -- (bcm2);
\draw[thin] (bcm) -- (acm2);
\draw[thin] (acm) -- (abm2);

\node at (acm) [anchor=north west]{$a'$};
\node at (abm) [anchor= east]{$b'$};
\node at (bcm) [anchor=south west]{$c'$};

\node at (abm2) [anchor=south east]{$a''$};
\node at (bcm2) [anchor=west]{$b''$};
\node at (acm2) [anchor= north east]{$c''$};

\end{tikzpicture}
\caption{Ideal $N$-morse triangles are slim}
\label{fig:slim}
\end{figure}

\begin{lemma}[Morse Triangles]\label{lem:morse triangles}
	Given $N$ and a triangle $T(a,b,c)$ with vertices in $X \cup \mbX,$ there exists $N'$ such that if two sides $[a,b]$, $[a,c]$ are $N$-Morse, then the third side $[b,c]$  is $N'$-Morse.
\end{lemma}

\begin{proof}
	For triangles with vertices in $X$, this is proved in Lemma 2.3 of \cite{Co}.  For an ideal triangle, following the proof of the Slim Triangle property above, we can truncate $T$ to form the triangle $R$ with vertices $a', b', c'$ and show that $[a', b']$ and $[a',c']$ are $N_1$-Morse.  Hence by Lemma 2.3 of \cite{Co}, $[b',c']$ is $N_2$-Morse, where $N_1, N_2$ depend only on $N$.  
	
Now apply Lemma 2.3 of \cite{Co} to the peripheral triangle formed by $b',c',c''$.  Since $[b',c']$ is $N_2$-Morse and $[c',c'']$ is length at most $L$, there exists $N_3$ depending only on $N_2$ and $L$, such that $[b',c'']$, the side contained in $[b,c]$, is $N_3$-Morse.  Since $b'$ and $c''$ can be chosen to be arbitrarily far out toward $b$ and $c$, it follows that $[b,c]$ is $N_3$-Morse.
\end{proof}

\subsection{The Morse boundary}

 We are now ready to define the Morse boundary.  
 
For two Morse rays $\alpha, \beta$ in $X$, say $\alpha \sim \beta$  if they have bounded Hausdorff distance 
and denote the equivalence class by $[\alpha]$.   The Morse boundary of $X$ consists of the set of equivalence classes of Morse rays.  To topologize this set, first choose a basepoint $p \in X$ and let $N$ be a Morse guage.  Set 
\begin{equation*} \mbdX_p= \{[\alpha] \mid \exists \beta \in [\alpha] \text{ that is an $N$--Morse geodesic ray with } \beta(0)=p\} \end{equation*} 
with the compact-open topology. These spaces are compact.  This topology is equivalent to one defined by a system of neighborhoods, $\{V_n(\alpha) \mid n \in \N \}$, defined as follows: $V_n( \alpha)$ is the set of $[\gamma] \in \mbdX_p$ such that   $d(\alpha(t), \gamma(t))< C_N$ for all $t<n$, where $C_N$ is the constant from Lemma \ref{lem:equiv geodesics}.

Let $\mathcal M$ be the set of all Morse gauges. Put a partial ordering on $\mathcal M$ so that  for two Morse gauges $N, N' \in \mathcal M$, we say $N \leq N'$ if and only if $N(\lambda,\epsilon) \leq N'(\lambda,\epsilon)$ for all $\lambda,\epsilon \in \N$.  Define the \emph{Morse boundary} of $X$ to be
 \begin{equation*} \mbX_p =\varinjlim_\mathcal{M} \mbdX_p 
 \end{equation*} 
with the induced direct limit topology, i.e., a set $U$ is open in $\mbX$ if and only if $U \cap \mbdX_p$ is open for all $N$.   It is shown in \cite{Co} that a change in basepoint results in a homeomorphic boundary, thus we normally omit the basepoint from the notation and denote the Morse boundary by $\mbX$.   

An alternate construction of the Morse boundary is given by the second author and Hume in \cite{cordes-hume}.   Define $X^{(N)}_p$ to be the set of all $y\in X$ such that there exists a $N$--Morse geodesic $[p,y]$ in $X$. By Proposition 3.2 in \cite{cordes-hume}, we know $X^{(N)}_p$ is $8N(3,0)$--hyperbolic in the Gromov 4-point definition of hyperbolicity.
Hence, we may consider its Gromov boundary, $\partial X^{(N)}_p$, and the associated visual metric $d_{(N)}$. We call the collection of boundaries $\left( \partial X^{(N)}_p, d_{(N)} \right)$ the \emph{metric Morse boundary} of $X$.  
(The reader is referred to \cite[Section 2.2]{BS} for a careful treatment of the sequential boundary of a $\delta$--hyperbolic space which is not necessarily geodesic.) 
 It is shown in \cite{cordes-hume} that there is a natural homeomorphism between $\partial X^{(N)}_p$ and $\mbdX_p$. 
 In particular, $\mbdX_p$ is compact.
 Thus the Morse boundary can also be defined as the direct limit of the metric spaces $\partial X^{(N)}_p$.

It follows from Lemma \ref{lem:morse triangles}, the Morse Triangle property,  that for any two points $a,b$ on $\mbdX_p$, there is a bi-infinite $N'$-Morse geodesic $\gamma$ connecting $a$ and $b$.  While the Morse guage $N$ depends on a choice of basepoint $p$, the Morse guage of the bi-infinite geodesic $\gamma$ does not.  In this paper we will primarily be concerned with the Morse guage for such bi-infinite geodesics.

Denote by $\mbX^{(n,N)}$, $n$-tuples of distinct points.  
$(a_1, a_2, \dots a_n)$ in $\mbX$ such that every bi-infinite geodesic from $a_i$ to $a_j$ is $N$-Morse. 
Let $(a,b,c) \in \mbX^{(3,N)}$.  By the Slim Triangle Property, for any ideal triangle $T(a,b,c)$, there exist points that lie within $\delta_N$ of all three sides of $T$.  While these points are by no means unique, we will show in the next Lemma that they form a bounded set.

\begin{lemma}\label{lem:E_K properties}
	 Let $(a,b,c) \in \partial_* X^{(3,N)}$. Set 
	 $$E_K(a,b,c)= \{ x \in X \mid \textrm{ $x$ lies within $K$ of all three sides of some $T(a,b,c)$}\}.$$  
	 For any  $K \geq \delta_N$, the following hold:
	\begin{enumerate}
		\item $E_K(a,b,c)$ is non-empty \label{center exists}
		\item $E_K(a,b,c)$ has bounded diameter $L$ depending only on $N$ and $K$.  \label{center bounded}
		\item For each vertex $v$ of $T(a,b,c)$, there exists points $p$ and $q$ on the two sides of $T$ emanating from $v$ such that $p,q \in E_K(a,b,c)$, $d(p,q) \leq \delta_N,$ and the Hausdorff distance from $[p,v]$ to $[q,v]$ is at most $\delta_N$. \label{ends eventually close}
	\end{enumerate}
\end{lemma}

\begin{proof}
	(\ref{center exists}): This follows from the Slim Triangle Property.  
		
	(\ref{center bounded}):  Lemma 3.1.5 of \cite{Bowditch91} proves this statement for triangles which are $\delta_N$-slim and have finite edge lengths. So what is left to show is that this holds for ideal $\delta_N$-slim triangles. 
	
	Let $p,q \in E_K(a,b,c)$. By Lemma \ref{lem:slim triangles} we can assume that $T(a,b,c)$ is $\delta_N$-slim. As in the proof of that lemma, we form the triangle $R$ with vertices $a',b',c'$ in $X$. The sides of $R$ are bounded distance from the sides of $T$ with the bound $K'$ depending only on $N$.  It follows  that $p$ and $q$ are within $K+K'$ of all three sides of $R$.  Thus, since the edge lengths of $R$ are finite, by Lemma 3.1.5 of \cite{Bowditch91} we know that the $d(p,q)$ is bounded where the bound depends only on $N$ and $K$. This completes the proof.

	(\ref{ends eventually close}): This follows from the continuity of the distance function and the fact that $T(a,b,c)$ is $\delta_N$-slim.
\end{proof}

We will refer to a point in $E_K(a,b,c)$ as a \emph{$K$-center} of $(a,b,c)$, or a \emph{coarse center} if $K$ is understood.

\section{Homeomorphisms induced by quasi-isometries}

Let $h:X \to Y$ be a quasi-isometry between proper geodesic spaces.  For a Morse ray $\alpha$ in $X$, the Morse property guarantees that the quasi-ray $h(\alpha)$ in $Y$ can be straightened to a Morse ray $\beta$ at bounded Hausdorff distance from  $h(\alpha)$.   In \cite{ChMu} and \cite{Co} the authors showed that the resulting map $\mbh: \mbX \to \mbY$ is a homeomorphism. In this section we will show that these homeomorphisms satisfy some additional properties.  


\subsection{Two-stable maps}\label{2-stable}
Recall that $\mbdX$ was defined in terms of a fixed basepoint $p$.  
While changing the basepoint does not change the set of points on the Morse boundary, it does change 
the Morse guage associated to any given point.
In this paper, on the other hand, we are concerned primarily with bi-infinite geodesics, rather than geodesic rays at a basepoint.
Let $\alpha$ be a bi-infinte, $N$-Morse geodesic in $X$.
While its endpoints $\alpha^+$ and $\alpha ^-$ are in $\mbX$, the rays from $x_0$ to these points may require much larger Morse guages, so $\alpha^+$ and $\alpha ^-$ need not lie in $\mbdX$.  

Since $\mbX$ has at least 3 points and the bi-infinite geodesic between any two of these is Morse,  $\mbX^{(3,N)}$ is non-empty for $N$ sufficiently large. 

\begin{definition}  Let $X$ and $Y$ be proper, geodesic metric spaces. A map $f : \mbX \to \mbY$ is \emph{2-stable} if for every Morse gauge $N$, there exists a Morse gauge $N'$ such that $f$ maps $\mbX^{(2,N)}$ into $\mbY^{(2,N')}$.  Note that it follows that $f$ maps $\mbX^{(n,N)}$ into $\mbY^{(n,N')}$ for all $n \geq 2$.  
\end{definition}

In the setting of CAT(0) spaces, one can instead use the contracting property to define 2-stable. Recall that a geodesic $\alpha$ is \emph{$D$-contracting} if the projection on $\alpha$ of any metric ball $B$, with $B \cap \alpha = \emptyset$, has diameter at most $D$.  It is shown in \cite{CS} that in a CAT(0) space, given $D$, there exists $N$ such that if $\alpha$ is $D$-contracting, then it is $N$-Morse, and conversely, given $N$ there exists $D'$ such that if $\alpha$ is $N$-Morse then it is $D'$ contracting.  It follows that for $X,Y$ CAT(0), $f : \mbX \to \mbY$ is 2-stable if and only if for each $D$ there exists $D'$ such that $f$ maps bi-infinite $D$-contracting geodesics to bi-infinite $D'$-contracting geodesics.  

Now suppose $f : \mbX \to \mbY$  is a homeomorphism. Since a closed set in $\mbX$ is compact if and only if it is contained in $\mbdX$, for some $N$ (see Lemma 4.1 of \cite{cordes-durham}), it  must be the case that for each $N$, there exists a $N'$ such that $f$ maps  $\mbdX$ into $\mbeY$.  On the other hand, this is does not guarantee that $f$ is 2-stable as the following example shows. 

\begin{example} \label{not 2-stable} Let $X$ be the Euclidean plane $\R^2$ with a ray $r_{m,n}$ attached at each lattice point $(m,n) \in \Z^2 \subset \R^2$. View the plane as horizontal and the attached rays as vertical.   It is easy to see that the Morse boundary is the discrete set of the vertical rays. 

Since $X$ is CAT(0), we can use the contracting property in place of the Morse property.
For a bi-infinite geodesic between two boundary points $r_{m,n}$ and $r_{s,t}$,  the optimal contracting constant for the bi-infinite geodesic connecting them is given by the distance in the plane from $(m,n)$ to $(s,t)$.

Consider the homeomorphism $f: \mbX \to \mbX$ which interchanges $r_{n,0}$ and $r_{-n,0}$  and leaves all other points on the boundary fixed.  Let $\alpha_n$ be the bi-infinite geodesic from $r_{n,0}$ to 
$r_{n,1}$. Then for all $n$, $\alpha_n$ is 1-contracting, whereas after applying $f$,  the bi-infinite geodesic between the resulting points, $r_{-n,0}$ and $r_{n,1}$ is worse than $2n$-contracting. Thus, $f$ is not 2-stable.

One can promote this example to a space  with a cocompact group action.  Namely, let $X$ be the universal cover of a torus wedge a circle, $T^2 \vee S^1$.  View the flats in $X$ as horizontal and the edges covering the circle as vertical.  Choose a base flat $F$ and identify it with $\R^2$.  Let $e(n,m)$ denote the upward edge attached at $(n,m) \in F$.  Define  $f: \mbX \to \mbX$ by interchanging any ray from the origin passing through   $e(n,0)$ with the corresponding ray passing through $e(-n,0)$, and leaving the rest of the boundary fixed.  This again defines a homeomorphism on the boundary which, by the same argument as above, is not 2-stable. 
\end{example}

\subsection{Cross-ratios}\label{cross-ratios}

We begin by reviewing Paulin's definition of the cross-ratio.
For four points $a,b,c,d$ in a $\delta$-hyperbolic space $X$, Paulin defines the cross-ratio to be 
$[a,b,c,d]=\frac{1}{2} (d(a,d) + d(b,c) - d(a,b) - d(c,d))$.
He then extends this definition to $\partial X$ by taking limits over sequences of points approaching the boundary.  We will use a slightly different definition of the cross-ratio motivated by the following observation. 
Since triangles in $X$ are $\delta$-slim, there exist points $p$ and $q$ lying within $\delta$ all three sides of the triangles $(a,b,c)$ and $(a,c,d)$ respectively.  It is easy to see that (the absolute value of) Paulin's cross-ratio is approximately equal to $d(p,q)$; they differ by at most $4\delta$.  
Analogous points also exist for ideal triangles, namely course centers of $(a,b,c)$ and $(a,c,d)$, and give rise to  a definition of the cross-ratio in $\mbX$ that is  intrinsic to the boundary and easier to work with.

\begin{figure}[h]

\begin{tikzpicture}[thick]

\newlength\mylen

\tikzset{
bicolor/.style n args={3}{
  decoration={
    markings,
    mark=at position 0.5 with {
      \node[draw=none,inner sep=0pt,fill=none,text width=0pt,minimum size=0pt] {\global\setlength\mylen{\pgfdecoratedpathlength}};
    },
  },
  draw=#1,
  dash pattern=on #3\mylen off 1\mylen,
  preaction={decorate},
  postaction={
    draw=#2,
    dash pattern=on (1-#3)*\mylen off (#3)*\mylen,dash phase=(1-#3)*\mylen
  },
  }
}

\tikzset{
tricolor/.style n args={5}{
  decoration={
    markings,
    mark=at position 0.5 with {
      \node[draw=none,inner sep=0pt,fill=none,text width=0pt,minimum size=0pt] {\global\setlength\mylen{\pgfdecoratedpathlength}};
    },
  },
  draw=#1,
  dash pattern=on #4\mylen off \mylen,
  preaction={decorate},
  postaction={
    draw=#2,
    dash pattern=on (#5-#4)*\mylen off \mylen,dash phase=(-#4)*\mylen
  },
  postaction={
    draw=#3,
    dash pattern=on (1-#5)*\mylen off \mylen,dash phase=(-#5)*\mylen
    },
  }
}

	\node[scale=0.4,circle,fill](a) at (0,0) {};
	\node at (a) [anchor=north east]{$a$};
	\node[scale=0.4,circle,fill](b) at (2,5) {};
	\node at (b) [anchor=south east]{$b$};
	\node[scale=0.4,circle,fill](c) at (12,3) {};
	\node at (c) [anchor=south west]{$c$};
	\node[scale=0.4,circle,fill](d) at (11,-3) {};
	\node at (d) [anchor=north]{$d$};
	\node (up) at (0,0.2) {};
	\node (down) at (0,-0.2) {};
	\node (right) at (0.2,0) {};
	\node (left) at (-0.2,0) {};
	\draw[-] (a) -- (c) node[pos=0.25,scale=0.4] (p){} node[pos=0.65,scale=0.4] (q){};
	\draw[bicolor={cyan}{red}{0.45}] (a)  ..controls ($(p)+1.5*(up)+(left)$) .. (b) node[pos=0.23,anchor = south east]{$-$} node[pos=0.77,anchor = east]{$-$};
	\draw[tricolor={red}{black}{cyan}{0.29}{0.65}] (b) ..controls ($(p)+(up)$) .. (c) node[pos=0.2,anchor = west]{$+$} node[pos=0.92,anchor = south]{$+$};
	\draw[tricolor={cyan}{black}{red}{0.25}{0.58}] (a) ..controls ($(q)+2*(down)+(right)$) .. (d) node[pos=0.08,anchor = north]{$+$} node[pos=0.75,anchor = north east]{$+$};
	\draw[bicolor={red}{cyan}{0.58}] (d) ..controls ($(q)+(down)$) .. (c) node[pos=0.22,anchor = west]{$-$} node[pos=0.82,anchor = north]{$-$};

	\node at ($(p)+3*(up)$)[fill,scale=0.4,circle](c1){};
	\node at ($(q)+3*(down)$)[fill,circle,scale=0.4](c2){};

	\node at (c1)[anchor=south east]{$p$};
	\node at (c2)[anchor=north west]{$q$};

\end{tikzpicture}
\caption{Distance between centers is coarsely equal to Paulin's cross-ratio}
\label{fig:project}
\end{figure}
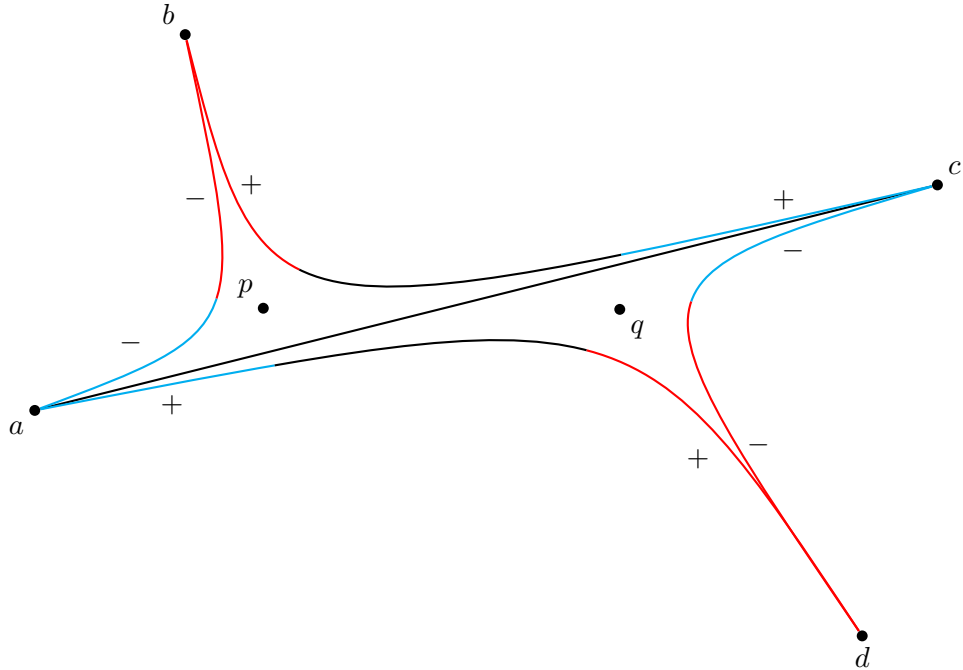

By the Slim Triangle Property (Lemma \ref{lem:slim triangles}), ideal triangles with vertices 
$(a,b,c) \in \mbX^{(3,N)}$ are $\delta_N$-slim, thus they contain points $p$ that lie within $\delta_N$ of all three sides.  By Lemma \ref{lem:E_K properties}, the set $E_{\delta_N}(a,b,c)$ of all such points is bounded.

\begin{definition}  Fix a Morse gauge $N$ and set  $K=\delta_N$.  
The \emph{cross-ratio} of a four-tuple $(a,b,c,d) \in \mbX^{(4,N)}$ is defined to be
$$ [a,b,c,d] = \pm \sup \{d(p,q) \mid p \in E_K(a,b,c),\,  q \in E_K(a,c,d)\} $$  
\end{definition}

Since we will only be concerned with the absolute value of the cross-ratio, we will not bother to specify the sign.  (Intuitively, the sign should depend on whether $p$ lies to the left or right of $q$ in Figure \ref{fig:project}, but making this precise is messy and of no use to us here.)
Note that interchanging $a$ with $c$, or $b$ with $d$, does not change the absolute value of the cross-ratio. 
Indeed, up to a uniform bound depending only on $N$, $|[a,b,c,d]|$ depends only on the pairing of $a$ with $c$, and $b$ with $d$.  To see this, approximate $(a,b,c,d)$ by points in the interior of $X$, and note that for interior points, Paulin's original formula for the cross-ratio differs from ours by at most $4\delta_N$.

\begin{lemma}\label{lem:choose centers}  Let $(a,b,c,d) \in \mbX^{(4,N)}$ and $K\geq \delta_N$.  Then there exists a constant $L=L(N,K)$, such that  for any choice of  $p \in E_K(a,b,c)$ and $q \in E_K(a,c,d),$
$$d(p,q) \leq  |[a,b,c,d]| \leq d(p,q) + 2L$$
\end{lemma}

\begin{proof}  This follows immediately  from Lemma \ref{lem:E_K properties}(\ref{center bounded}).
\end{proof}

\begin{remark}
As the reader may have noticed, $\delta_N$, and hence the definition of the cross-ratio, depends on the choice of $N$.  If one increases the Morse gauge to $N' >N$, then any choice of coarse centers $p$ and $q$ for $K=\delta_N$, are also coarse centers for $K'=\delta_{N'}$, hence by the lemma, the cross-ratio will change by at most a bounded amount.   Since we will primarily be concerned with bounding cross-ratios of 4-tuples $(a,b,c,d) \in \mbX^{(4,N)}$ for a fixed $N$, and most arguments will be done using a choice of coarse centers, this ambiguity will be of no concern.   

It is possible to modify our definition of the cross-ratio to make it independent of $N$ by choosing a
$K$ for each triple of vertices that depends on the vertices, but not on $N$.  For example, we could set
$$K(a,b,c) = 1 + \inf \{ k \mid E_{k}(a,b,c) \neq \emptyset \},$$ 
and then define the cross-ratio as the supremum of $d(p,q)$ where $p$ is a $K(a,b,c)$-center and $q$ is a $K(a,c,d)$-center.  For $(a,b,c,d) \in \mbX^{(4,N)}$, $E_{\delta_N}$ is always non-empty, so this definition of the cross-ratio would differ from the one above by a bounded amount with the bound depending only on $N$. 

In an earlier version of this paper \cite{ChMu}, where we restricted our attention to CAT(0) spaces, we took $p$ and $q$ to be projections of $b$ and $d$ on a geodesic $[a,c]$, rather than coarse centers, and used these to define the cross-ratio.   These projections, and hence this definition of the cross-ratio, are independent of $N$.  This works since in a CAT(0) space, Morse geodesics are (strongly) contracting and projections are well-defined (see \cite{Mu}).  Unfortunately, this approach does not generalize to Morse geodesics in more general geodesic metric spaces. 
\end{remark}

\begin{definition}  Let $X$ and $Y$ be geodesic metric spaces.  A homeomorphism $f: \mbX \to \mbY$ is  
\emph{$N$-quasi-mobius}  if there exists a continuous map 
 $\psi_N : [0,\infty) \to [0,\infty)$  such that for all 4-tuples $(\alpha,\alpha',\beta,\beta')$ in $\mbX^{(4,N)}$, 
$$ |[f(\alpha),f(\alpha'),f(\beta),f(\beta')]| \leq \psi_N(|[\alpha,\alpha',\beta,\beta']|). $$
We say that $f$ is \emph{quasi-mobius} if it is $N$-quasi-mobius for every $N$.   
\end{definition} 

We remark that one can always choose the functions $\psi_N$ to be non-decreasing.			
			
\begin{theorem}\label{forward} Let $h: X \to Y$ be a $(\lambda,\epsilon)$-quasi-isometry between two proper geodesic metric spaces.  Then the induced map $\mbh : \mbX \to \mbY$ is a 2-stable, quasi-mobius homeomorphism.  Moreover, the functions $\psi_N$ in the definition of quasi-mobius can all be taken to be linear with multiplicative constant $\lambda$.
 \end{theorem}

\begin{proof}  The fact that $\mbh$ is a homeomorphism was proved by the first author and H.~Sultan for CAT(0) spaces in \cite{CS} and then generalized to arbitrary proper geodesic metric spaces by the second author in \cite{Co}.  The proof involves showing that for each Morse guage $N$,  there exists a Morse guage $N'$ such that the image of an $N$-Morse ray under the quasi-isometry $h$ can be ``straightened"  to an $N'$-Morse ray in $Y$.  The same proof applies to bi-infinite geodesics to show that $\mbh$ is 2-stable. 

Suppose $h$ is a $(\lambda, \epsilon)$-quasi-isometry.  To prove that $\mbh$ is quasi-mobius, first consider a triple $(a,b,c) \in \mbX^{(3,N)}$ and let $(a',b',c') \in \mbY^{(3,N')}$  be its image under $\mbh$.  Let  $T=T(a,b,c)$ be a representative triangle.  Applying $h$ to $T$ gives a quasi-triangle (a triangle whose sides are quasi-geodesic) in $Y$.  The sides can be straightened to geodesics to obtain a triangle $T'= T'(a',b',c')$ with $N'$-Morse sides.  The Morse property guarantees that there is a constant $C$, depending only on $N', \lambda, \epsilon$,  such that  $h(T)$ lies in the $C$-neighborhood of $T'$.  

Set $K=\delta_N$ and consider the image of $E_K(a,b,c)$ under $h$.  For any $x \in E_K(a,b,c)$,  $h(x)$ lies within $\lambda K +\epsilon$ of all three sides of $h(T)$ for some $T$ and hence within  $\lambda K +\epsilon+C$ of all three sides of $T'$.  Taking $K'=\max\{\delta_{N'}, \lambda K +\epsilon+C\}$, we conclude that the image of $E_K(a,b,c)$ under $h$ lies in $E_{K'}(a',b',c')$.  By Lemma \ref{lem:E_K properties}(\ref{center bounded}), the diameter of $E_K(a,b,c)$ is bounded by a constant $L=L(N,K)$ and the diameter of $E_{K'}(a',b',c')$ is bounded by a constant $L'=L'(N',K')$.  

Now consider a 4-tuple $(a,b,c,d) \in X^{(4,D)}$ and let $(a',b',c',d') \in Y^{(4,D')}$ be its image under $\mbh$.  Let $p$ and $q$ be $K$-centers of $(a,b,c)$ and $(a,c,d)$ respectively.  Likewise, let $p',q'$ be $K'$-centers of $(a',b',c')$ and $(a',c',d')$.    By Lemma \ref{lem:choose centers},  $|[a,b,c,d]|$ and $d(p,q)$ differ by at most $2L$  and $|[a',b',c',d']|$ and $d(p',q')$ differ by at most $2L'$.  Hence to prove $\mbh$ is quasi-mobius, it suffices to bound $d(p',q')$ as a function of $d(p,q)$.

As observed above, $h(p) \in E_{K'}(a',b',c')$, so $d(p',h(p)) < L'$, and likewise, $d(q',h(q)) < L'$.  Thus,
$$d(p',q') < d(h(p),h(q)) + 2L' \leq \lambda d(p,q) + \epsilon +2L'.$$
It follows that $\mbh$ is quasi-mobius with linear bounding functions with multiplicative constant $\lambda$.
\end{proof}

\section{Quasi-isometries induced by homeomorphisms}

The  goal of this section is to prove a converse of Theorem \ref{forward}.  
Namely, if $f: \mbX \to \mbY$ is a homeomorphism such that $f$ and $f^{-1}$ are both 2-stable and quasi-mobius, then $f$ is induced by a quasi-isometry $h: X \to Y$.  We continue to assume that $\mbX$ (and hence $\mbY$) contains at least three points.  In addition, to prove the converse statement we will need to assume that both $X$ and $Y$ are cocompact,  that is, there exists a group $G$ acting cocompactly on $X$ and a group $H$ acting cocompactly on $Y$.

\subsection{Extending \texorpdfstring{$f$}{f} to the interior}
Given a homeomorphism $f:\mbX \to \mbY$, we need to extend it to a map $h: X \to Y$.
Let us outline the steps involved in defining such an $h$. 

 Choose $N$ such that $ \mbX^{(3,N)}$ is non-empty.  We begin by defining a map  
$$ \pi^N_X: \mbX^{(3,N)} \to X.$$  
Fix  $K = \delta_N$.
Define  $\pi^N_X (a,b,c) = p$ where $p$ is a choice of $K$-center for $(a,b,c)$, 
that is, $p \in E_K(a,b,c)$.  
When $N$ is fixed, we will generally omit it from the notation and denote the map by $\pi_X$.  We will refer to
$\pi_X (a,b,c)$ as the \emph{projection} of $(a,b,c)$.

If $G$ is a group acting cocompactly by isometries  on $X$, then the induced action of $G$ on $\mbX$ preserves the Morse constants of bi-inifinte geodesics and we can choose $\pi_X$ to be equivariant with respect to the induced action.  Choose a basepoint $x_0$ that lies in the image of $\pi_X$.  Since the action of $G$ on $X$ is cocompact, there is a ball $B(x_0,R)$ whose $G$-translates cover $X$, so every point in $X$ lies within $R$ of $\pi_X(a,b,c)$ for some $(a,b,c)$.

Now assume that $f : \mbX \to \mbY$ is a 2-stable homeomorphism and say
$f (\mbX^{(2,N)}) \subseteq \mbY^{(2,N')}$. As observed above, it follows that  $f (\mbX^{(n,N)}) \subseteq \mbY^{(n,N')}$ for all $n \geq 2$.  Set $\pi_X=\pi^N_X$  and let 
$\pi_Y = \pi^{N'}_Y : \mbY^{(N',3)} \to Y$ be the analogous map for $Y$.  We will define $h(x)$ to be a point in the set
$$\Pi(x)= \pi_Y\circ f \circ \pi_X^{-1}(B(x,R)) \subset Y.$$ 
To have any control on this function $h$, we must first prove that for any $x$,  $\Pi(x)$ has bounded diameter, depending only on $N$.

\begin{lemma} \label{flips}   There exists a constant $C_1$ depending only on $N$ such that for any  
$(a,b,c,d) \in \mbX^{(4,N)}$, the absolute value of one of the three cross-ratios $[a,b,c,d]$, $[a,c,b,d]$, 
$[c,a,b,d]$ is less than $C_1$.
\end{lemma}

\begin{figure}

\begin{tikzpicture}[thick]

	\node (up) at (0,0.2) {};
	\node (down) at (0,-0.2) {};
	\node (right) at (0.2,0) {};
	\node (left) at (-0.2,0) {};

	\node (a) at (0,0) {};
	\node at (a) [anchor=north east]{$a$};
	\node (b) at (2,5) {};
	\node at ($(b)+3*(up)+(left)$) {$b$};
	\node (c) at (12,3) {};
	\node at (c) [anchor=south west]{$c$};
	\node (d) at (11,-3) {};
	\node at (d) [anchor=north]{$d$};
	\draw[<->] (a) -- (c) node[pos=0.25,scale=0.4,circle,fill] (p){} node[pos=0.65,scale=0.4,circle,fill] (q){};
	\draw[<->] ($(a)+2*(up)+0.5*(right)$)  ..controls ($(p)+1.5*(up)+2.5*(left)$) .. ($(b)+1.5*(left)$) node [pos=0.45,scale=0.4,circle,fill](p2){} node at(p2)[anchor=south east]{$p''$};
	\draw[<->] ($(b)+1.5*(right)+(up)$) ..controls ($(p)+1.5*(up)+2*(right)$) .. ($(c)+(up)+(left)$) node [pos=0.3,scale=0.4,circle,fill](p1){} node at(p1)[anchor=south west]{$p'$};
	\draw[<->] ($(a)+2*(down)+(right)$) ..controls ($(q)+3*(down)+2*(left)$) .. ($(d)+2*(left) + 0.5*(up)$) node [pos=0.175,circle,fill,scale=0.4](r){} node at(r)[anchor=north west]{$r$} node [pos=0.6,scale=0.4,circle,fill](q1){} node at (q1)[anchor = north east]{$q'$};
	\draw[<->] ($(d)+(right)+(up)$) ..controls ($(q)+(down) + 3*(right)$) .. ($(c)+1.5*(down)+0.5*(left)$);
	\draw[dashed,<->] ($(b)+0.5*(up)$) .. controls ($(p)+20*(down)+ 2*(right)$) and ($(q)+20*(up) + 10*(left)$) .. (d) node [pos=0.2,scale=0.4,circle,fill](r1){} node at (r1)[anchor=north east]{$r'$};


	\node at (p)[anchor=south east]{$p$};
	\node at (q)[anchor=north west]{$q$};

\end{tikzpicture}
\caption{At least one cross-ratio is ``small".}
\label{fig:small cross-ratio}
\end{figure}
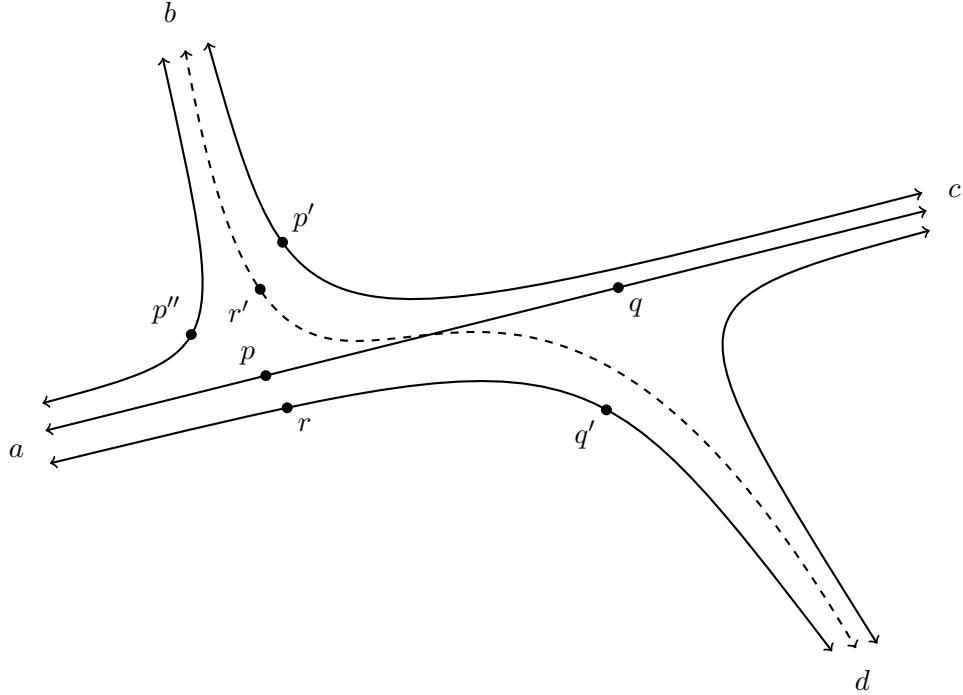

\begin{proof}
	Let $K=\delta_N$.  Let $L$ be the bound on $\diam(E_K)$ given by Lemma \ref{lem:E_K properties}(\ref{center bounded}) and $L'$ the bound on $\diam(E_{L+2K})$. We may assume $L' > L$ (otherwise replace $L'$ by $L+1$). Note that since $L$ depends only on $N$,  $L'$ also only depends only on $N$.  Set $C_1=2L'$.

	If $|[a,b,c,d]|< C_1$ then we are done. For the rest of the proof, suppose that $|[a,b,c,d]|\geq C_1$. 
Note that in this case $E_K(a,b,c) \cap E_K(a,c,d)$ is empty since for any $p \in E_K(a,b,c)$ and $q \in E_K(a,c,d)$,
$$ d(p,q) \geq |[a,b,c,d]|-2L \geq C_1 -2L > 0.$$
Interchanging $a$ and $c$ if necessary, we may assume that $E_K(a,b,c)$ lies ``closer" to $a$ than $E_K(a,c,d)$.  That is, for any points $p \in E_K(a,b,c) \cap [a,c]$ and $q \in E_K(a,c,d) \cap [a,c]$,
$p$ lies on the ray $[a,q]$. See Figure \ref{fig:small cross-ratio}.

By  Lemma \ref{lem:E_K properties}(\ref{ends eventually close}),  there exist points $p, p' \in E_{K}(a,b,c)$ with  $p \in [a,c]$ and $p' \in [b,c]$, such that $[p,c]$ and $[p',c]$ have Hausdorff distance at most $K$.  Likewise, there are points $q, q' \in E_K(a,c,d)$ with $q \in [a,c]$ and $q' \in [a,d]$ such that $[a,q]$ and $[a,q']$ have Hausdorff distance at most $K$. In particular, since $p \in [a,q]$, there exists a point $r \in [a,q'] \subset [a,d]$ with $d(p,r) \leq K$. 

Now consider the triangle $(b,c,d)$.  By the slim triangle property, $p' \in [b,c]$ lies within $K$ of a point $r'$ on one of the other two sides.  Say $r' \in [c,d]$.   Since $d(p,p') \leq L$, the points $p,r,r'$ all lie within $L+2K$ of each other, hence they lie in $E_{L+2K}(a,c,d)$.  Since $p$ also lies in $E_K(a,b,c)$, Lemma \ref{lem:choose centers}   implies that  $|[a,b,c,d]| \leq 2L' = C_1$.  But this contradicts our original assumption.

Thus we must have $r' \in [b,d]$.  Choose a point $p'' \in [a,b] \cap E_K(a,b,c)$.  Then $p'', r, r'$ all lie within $L+2K$ of each other, hence they lie in $E_{L+2K}(a,b,d)$.   Since $p''$ also lies in $E_K(a,b,c)$, this implies that
$|[a,c,b,d]| \leq 2L' = C_1$.

Recalling that we initially allowed $a, c$ to be interchanged if necessary, we have thus shown that one of $|[a,b,c,d]|,
|[a,c,b,d]|$, or $|[c,a,b,d]|$ is less than $C_1$.
\end{proof}

By Lemma \ref{lem:choose centers}, we see that if $|[a,b,c,d]|$ is small, then any pair of $K$-centers of the triangles $(a,b,c)$ and $(a,c,d)$ are close.  Thus, the lemma above says that for $(a,b,c,d) \in \mbX^{(4,N)}$, replacing one of the vertices of the triangle $(a,b,c)$ by $d$, causes only a small  change in $\pi_X(a,b,c)$.

\begin{proposition} \label{bndd expansion1}   Let $f : \mbX \to \mbY$ be a 2-stable, quasi-mobius homeomorphism.  Then for any Morse guage $N$, there exists a function $C_N: \R^+ \to \R^+$ such that 
for all $(a,b,c), (u,v,w)$ in $\mbX^{(3,N)}$, 
$$d_X(\pi_X(a,b,c), \pi_X(u,v,w)) \leq R \Longrightarrow\\  d_Y(\pi_Y(f(a),f(b),f(c)), \pi_Y(f(u),f(v),f(w)) \leq C_N(R).$$
\end{proposition}

\begin{proof} Let $T_1=(a,b,c)$ and $T_2=(u,v,w)$ and set $x=\pi_X(a,b,c)$ and $y= \pi_X(u,v,w)$.  Fix a constant $M\geq 0$.  We first show that there exists a Morse gauge $N'$, depending only on $N$ and $M$, such that if $d(x,y) \leq M$, then any geodesic from a vertex of $T_1$ to a vertex of $T_2$ is $N'$-Morse.  Say, for example, that $\gamma \in (a,u)$ is such a geodesic.  By Lemma \ref{lem:E_K properties}, $x$ (resp. $y$) is uniformly bounded distance, say distance $C$, from the sides of any triangle representing $T_1$ (resp. $T_2$).  It follows that
$(a,x)$ is in the $C$-neighborhood of any geodesic in $(a,b)$, and $(u,y)$ is in the $C$-neighborhood of any geodesic in $(u,v)$.
Assuming $d(x,y) \leq R$ this also implies that $(u,x)$ is in the $(C+R)$-neighborhood of $(u,v)$.  It follows that there is a Morse gauge $N_1$, depending only on $C$ and $R$, such that $(a,x)$ and $(u,x)$ are $N_1$-Morse, and hence by the Morse Triangles Property, a Morse gauge $N'$ such that $[a,u]$ is $N'$-Morse.  The same argument applied to other pairs of vertices shows that $(a,b,c,u,v,w) \in \mbX^{(N',6)}$.  

By Lemma \ref{flips}, there is a constant $C_1$ such that for any $(p,q,r,s) \in \mbX^{(N',4)}$, some permutation of the first three points $p,q,r$ results in a cross-ratio with absolute value bounded by $C_1$.  We will say that such a cross-ratio is ``small".   It follows from Lemma \ref{lem:choose centers}, that if $|[p,q,r,s]|$ is small, then  $\pi_X(p,q,r)$ and $\pi_X(p,s,r)$ are uniformly close, say at distance $< C_2$.  In this case, we call the move from $(p,q,r)$ to $(p,s,r)$ a ``small flip".  So for any $(p,q,r,s) \in \mbX^{(N',4)}$, there exists a small flip that replaces one vertex of the triangle $(p,q,r)$ by $s$. 

To prove the proposition, we begin by showing that applying at most 3 small flips to the triangles $T_1=(a,b,c)$ and $T_2=(u,v,w)$, we obtain a pair of triangles that share an edge.   We first do a flip that replaces a vertex of $(a,b,c)$ by $v$. Permuting $(a,b,c)$ if necessary, we get a small flip from $(a,b,c)$ to $(a,v,c)$.  Next, we replace a vertex of $(a,v,c)$ by $w$.   If the flip to either $(w,v,c)$ or $(a,v,w)$ is small, we are done since these share an edge with $(u,v,w)$.  

So suppose only the flip from $(a,v,c)$ to $(a,w,c)$ is small.
In this case, consider the flips of $(u,v,w)$ obtained by replacing a vertex by $a$.  All of the resulting triangles $(a,v,w)$, $(u,a,w)$, $(u,v,a)$ share an edge with either $(a,v,c)$ or $(a,w,c)$, so whichever one of these flips is small, we arrive at the desired pair of adjacent triangles. 

Since the projections $x,y$ of $T_1,T_2$ are at distance at most $R$, the projections of the resulting pair of adjacent triangles are at distance at most $R'=R+3C_2$.  
In summary, there is a sequence of at most 5 triangles with vertices in $\{a,b,c,u,v,w\}$, beginning with $T_1$ and ending with $T_2$ such that consecutive triangles share an  edge and have projections at distance at most $R'$.  

Now apply $f$ to this sequence of triangles. Since $f$ is 2-stable, $f(a,b,c,u,v,w)$ lies in $\mbY^{(N'',6)}$ for some Morse gauge $N''$.   
Using Lemma \ref{lem:choose centers} and the quasi-mobius function $\Psi_{N'}$, one gets a bound on the distance between the projections in $Y$ of consecutive triangles, and hence a bound on the distance between the projections of $f(T_1)$ and $f(T_2)$.  This proves the proposition.
\end{proof}

In particular, it follows that for any $x \in X$, the set $\pi_Y(f(\pi_X^{-1}(B(x,R)))) \subset Y$ has bounded diameter with the bound, say $M$, depending only on the choice of  $N, N'$ and $R$.  Set
$$\Pi (x) = \pi_Y(f(\pi_X^{-1}(B(x,R)))) $$
and define a map $h : X \to Y$ by choosing a point
$$ h(x) \in \Pi (x). $$ 
We call such an $h$ an extension of $f$ to $X$.
While the definition of $\Pi$ depends on a choice of $N, N'$ and $R$, increasing any of these constants just increases the size of $\Pi$, so the same choice of $h(x)$ works for these larger constants.

\begin{proposition} \label{bndd expansion2} Fix $N, N'$, and $R$ as above. Then for any $Q \geq 0$, there exists $C_3$ such that 
$$d_X(x,y) \leq Q \Longrightarrow  d_Y(h(x), h(y)) \leq C_3.$$
\end{proposition}  

\begin{proof}  Let $(a,b,c)$ and $(u,v,w)$ be points in $\mbX^{(3,N)}$ and whose projections $\pi_X(a,b,c)$ and $\pi_X(u,v,w)$ lie  in $B(x,R)$ and $B(y,R)$ respectively.  Suppose  $d_X(x,y) \leq Q$ and hence $d_X(\pi_X(a,b,c), \pi_X(u,v,w)) \leq Q+2R$.  Then by Proposition \ref{bndd expansion1}, there exists $C = C_N(Q+2R)$ 
such that  
$$d_Y(\pi_Y(f(a),f(b),f(c)), \pi_Y(f(u),f(v),f(w)) \leq C.$$
Now $\pi_Y(f(a),f(b),f(c))$ is an element of $\Pi(x)$ which is a set of diameter at most $M$, so  $h(x)$ lies within $M$ of  $\pi_Y(f(a),f(b),f(c))$.  Similarly, $h(y)$ lies within $M$ of  $\pi_Y(f(u),f(v),f(w))$.
We conlude that  $d_Y(h(x), h(y)) \leq C + 2M$.  
\end{proof}

\subsection{Main Theorem}  We are now ready to prove our main theorem.

\begin{theorem}\label{reverse} Let $X$ and $Y$ be proper, cocompact, geodesic metric spaces
and assume $\mbX$ has at least 3 points.
Suppose $f : \mbX \to \mbY$ is a 2-stable, quasi-mobius homeomorphism, and likewise for $f^{-1}$.  Then there exists a quasi-isometry  $h: X \to Y$  with $\partial_*h=f$.
\end{theorem}

\begin{proof}  
Choose $N'$ so that $f(\mbX^{(2,N)}) \subseteq \mbY^{(2,N')}$ and $f^{-1}(\mbY^{(2,N)}) \subseteq \mbX^{(2,N')}$.  Say  $G \curvearrowright X$ and $H  \curvearrowright Y$ are cocompact group actions and choose
 $R>0$ so that both $X$ and $Y$ are covered by the $R$-neighborhood of an orbit.  Choose projection maps $\pi_X$ and $\pi_Y$ that are equivariant with respect to the actions of $G$ and $H$ respectively.
 
Let $h$ be an extension of $f$ to $X$ and let  $h^{-1}$ to be an extension of $f^{-1}$ to $Y$ compatible with these constants.  That is,  $h(x) \in \Pi(x)$ and $h^{-1}(y) \in \Pi(y)$, where
$$\Pi(x)= \pi_Y(f (\pi_X^{-1}(B(x,R)))) $$ 
$$\Pi(y)= \pi_X(f^{-1}(\pi_Y^{-1}(B(y,R))))$$

To prove that $h$ is a quasi-isometry, it suffices to show, 
\begin{itemize}
\item[(i)] for all $x,y \in X$ and $p,q \in Y$ there are linear bounds $d_Y(h(x),h(y)) \leq A\, d_X(x,y) + B$
and $d_Y(h^{-1}(p),h^{-1}(q)) \leq A' d_Y(p,q) + B'$, and
\item[(ii)]  $h$ and $h^{-1}$ are quasi-inverses.
\end{itemize}

For (i), let $S$ be a finite generating set for $G$.  Choose a base point $x_0 \in X$.
Approximate the geodesic from $x$ to $y$ by a sequence of orbit points $x, g_0x_0, g_1x_0, \dots g_nx_0, y$ such that $g_{i+1}=g_is_i$ for some generator $s_i \in S$.  Now map this sequence by $h$ into $Y$.  Since the distance between consecutive points is bounded, Proposition \ref{bndd expansion2} implies that there exists $C_3$ such that 
$d_Y(h(x),h(y)) \leq C_3\,(n + 2)$.  Since the inclusion of $G$ into $X$ as the orbit of $x_0$ is a quasi-isometry, there exists $\lambda, \epsilon$ such that $n=d_G(g_0,g_n) \leq \lambda d_X(x,y) + \epsilon$, so 
$$d_Y(h(x),h(y)) \leq C_3\lambda\, d_X(x,y) + C_3 (\epsilon + 2).$$
An analogous argument for $h^{-1}$ gives an upper bound on $d_X(h^{-1}(p),h^{-1}(q))$ as a linear function of $d_Y(p,q)$.

Next we prove that $h$ and $h^{-1}$ are quasi-inverses.  Let $M$ denote an upper bound on the diameter of $\Pi(x)$.
Say $N''$ is such that $f^{-1}(\mbY^{(2,N')}) \subseteq \mbX^{(2,N'')}$. Choose $R' \geq \max\{R,M\}$ and note that $\Pi(x) \subset B(h(x), R')$ 
for all $x$.  Let  $\pi'_X= \pi^{N''}_X$ and set
$$\Pi'(y)= \pi'_X(f^{-1}(\pi_Y^{-1}(B(y,R')))). $$ 
Since $\Pi'$ is obtained from $\Pi$ simply by increasing the constants, $\Pi(y) \subseteq \Pi'(y)$ for all $y$, and  $\Pi'(y)$ has uniformly bounded diameter, say bounded by $M'$. 

Now suppose $y=h(x)$.   Then for any  $(a,b,c) \in \mbX^{(3,N)}$ whose projection $z=\pi_X(a,b,c)$ lies in $B(x,R)$, 
we have $\pi_Y(f(a,b,c)) \in \Pi(x) \subset B(y,R')$, and hence $z$ and $h^{-1}(y)$ both lie in $\Pi'(y)$.  Thus,
$$d(x,h^{-1}\circ h(x)) = d(x,h^{-1}(y)) \leq d(x,z) + d(z, h^{-1}(y)) \leq R + M'.$$
An analogous argument shows that  $h^{-1}\circ h$ is also bounded distance from the identity map. 

It remains to show that $\mbh =f$. Choose a basepoint $x \in X$ such that  $x = \pi_X(a,b,c)$ for some triple $(a,b,c) \in \mbX^{(N,3)}$.  Let $p$ be a point in $\mbX$ and let $\alpha$ be a geodesic ray from $x$ to $p$, so $\alpha$ is $N'$-Morse for some $N' \geq N$.   Approximate $\alpha$ by a sequence of orbit points $\{g_ix\}$  converging to $p$.  Set $x_i=g_ix$, so $x_i$  is the projection of the triple $(a_i, b_i, c_i) = (g_ia, g_ib, g_ic)$.  Note that there exists $N''$ such that the geodesic ray from $x$ to any vertex in any one of these ideal triangles is $N''$-contracting.  This follows from the Morse Triangle Property since the triangle $T(x,x_i,a_i)$ has two sides, $[x,x_i]$ and $[x_i,a_i]$, which stay close to $N'$-Morse rays, and similarly for $b_i$ and $c_i$.  

Consider the sequences $(a_i),(b_i)$ and $(c_i)$.  
Since $\partial_* X^{N''}$ is compact, passing to a subsequence if necessary, all three sequences converge.  Say 
they converge to points $A,B,C$ respectively.  We claim that two of these points must be equal.  If not, then $(A,B,C)$ lies in $\mbX^{(N'',3)}$, and the sequence $(x_i)$ converges to a point in $E_K(A,B,C)$.  But this is impossible since $(x_i)$ converges to the point $p$ on the boundary.  So suppose $A=B$.  Since all of the $x_i$ lie uniformly bounded distance from some geodesics from $a_i$ to $b_i$, the sequence $(x_i)$ must also converge to this same point and we conclude that $p=A=B$.

Since $f: \mbX \to \mbY$ is a homeomorphism, it follows that two of the sequences  $\{f(a_i)\}, \{f(b_i)\},\{f(c_i)\}$ converge to $f(p)$ and hence the projections $\{y_i\}$ of these triangles also converge to $f(p)$.  The point $p$ is represented by the geodesic ray $\alpha$, so $\partial_* h(p)$ is represented by a geodesic straightening of 
$h(\alpha)$.  Since $y_i$  lies uniformly bounded distance from $h(x_i)$, and hence also from $h(\alpha)$, we conclude that $\partial_* h(p) = \lim y_i = f(p)$.
\end{proof}

Combining Theorems \ref{forward} and \ref{reverse} we thus have
\begin{theorem}\label{main} Let $X$ and $Y$ be proper, cocompact geodesic metric spaces with at least 3 points in their Morse boundaries.   A homeomorphism $f : \mbX \to \mbY$ is induced by a quasi-isometry  $h : X \to Y$ if and only if $f$  and $f^{-1}$ are 2-stable and quasi-mobius.  
\end{theorem}

\begin{remark} In \cite{CaMa}, Cashen and Mackay define an alternate topology on the Morse boundary that has the advantage that it is metrizable.  Denote this boundary by $\partial_{CM}X$.  We claim that any continuous map $f : \partial_{CM}X \to \partial_{CM}Y$ that is 1-stable (i.e., for each $N$ there exists $N'$ such that  $f(\mbdX) \subseteq \partial_*^{N'}Y$)
 is also continuous when viewed as a map from $\mbX$ to $\mbY$.  This follows from the fact that the inclusion of $\mbdX$ into $\partial_{CM}X$ is a topological embedding, by Corollary 6.2 of  \cite{CaMa}, and hence $f : \mbdX \to \partial_*^{N'}Y$ is continuous for each $N$.  Thus, Theorem \ref{main}  also applies to any homeomorphism 
$f : \partial_{CM}X \to \partial_{CM}Y$ for which both $f$ and $f^{-1}$ are 1-stable.
\end{remark}

We conclude by noting that the homeomorphisms described in Example \ref{not 2-stable} that are not 2-stable, also fail to be quasi-mobius.  For example, setting $m=n+1$, the cross-ratio
$$|[r_{n,0},r_{m,0},r_{n,1},r_{m,1}]|=1$$
whereas after applying $f$ we get a cross-ratio of
$$|[r_{-n,0},r_{-m,0},r_{n,1},r_{m,1}]| > 2n-1$$

It seems reasonable ask to whether this is always the case, at least in the CAT(0) setting.
\begin{question}  Let $X$, $Y$ be as in the theorem and $f: \mbX \to \mbY$ a homeomorphism.  Does $f$ quasi-mobius imply that $f$ is 2-stable? 
\end{question}

\bibliographystyle{utphys}
\bibliography{quasi-mobius}

\providecommand{\href}[2]{#2}\begingroup\raggedright\begin{thebibliography}{10}

\bibitem{O}
J.~Otal, ``Sur la g{\'e}ometrie symplectique de l'espace des
  g{\'e}od{\'e}siques d'une vari{\'e}t{\'e} {\`a} courbure n{\'e}gative,'' {\em
  Revista Matem{\'a}tica Iberoamericana} {\bfseries 8} (1992) 441--456.

\bibitem{Pan}
P.~Pansu, ``Dimension conforme et sph{\`e}re {\`a} l'infini des
  vari{\'e}t{\'e}s {\`a} courbure n{\'e}gative.,'' {\em Annales Academi{\ae}
  Scientiarum Fennic{\ae}} {\bfseries 14} (1989) 177--212.
  \url{http://www.acadsci.fi/mathematica/Vol14/vol14pp177-212.pdf}.

\bibitem{T86}
P.~Tukia, ``On quasiconformal groups,''
  \href{http://dx.doi.org/10.1007/BF02796595}{{\em Journal d'Analyse
  Math{\'e}matique} {\bfseries 46} no.~1, (Dec, 1986) 318--346}.
  \url{http://dx.doi.org/10.1007/BF02796595}.

\bibitem{TV82}
P.~Tukia and J.~V{\"a}is{\"a}l{\"a}, ``Quasiconformal extension from dimension
  n to n + 1,'' {\em Annals of Mathematics} {\bfseries 115} no.~2, (1982)
  331--348. \url{http://www.jstor.org/stable/1971394}.

\bibitem{TV84}
P.~Tukia and J.~V{\"a}is{\"a}l{\"a}, ``Bilipschitz extensions of maps having
  quasiconformal extensions,'' \href{http://dx.doi.org/10.1007/BF01450765}{{\em
  Mathematische Annalen} {\bfseries 269} no.~4, (Nov, 1984) 561--572}.
  \url{http://dx.doi.org/10.1007/BF01450765}.

\bibitem{V}
J.~V{\"a}is{\"a}l{\"a}, {\em Lectures on n-Dimensional Quasiconformal
  Mappings}.
\newblock Lecture Notes in Mathematics. Springer Berlin Heidelberg.

\bibitem{Pau}
F.~Paulin, ``Un groupe hyperbolique est d{\'e}termin{\'e} par son bord,''
  \href{http://dx.doi.org/10.1112/jlms/54.1.50}{{\em Journal of the London
  Mathematical Society} {\bfseries 54} (1996) 50--74}.
  \url{http://dx.doi.org/10.1112/jlms/54.1.50}.

\bibitem{CK}
C.~Croke and B.~Kleiner, ``Spaces with nonpositive curvature and their ideal
  boundaries,'' {\em Topology} {\bfseries 39} no.~3, (2000) 549--556.

\bibitem{CS}
R.~Charney and H.~Sultan, ``Contracting boundaries of {CAT(0)} spaces,'' {\em
  Journal of Topology} {\bfseries 8} (2013) 93--117.

\bibitem{Co}
M.~Cordes, ``Morse boundaries of proper geodesic metric spaces,'' {\em Groups
  Geom. Dyn.} {\bfseries 11} no.~4, (2017) 1281--1306.
  \url{https://doi.org/10.4171/GGD/429}.

\bibitem{cordes-hume}
M.~Cordes and D.~Hume, ``Stability and the {M}orse boundary,'' {\em J. Lond.
  Math. Soc. (2)} {\bfseries 95} no.~3, (2017) 963--988.

\bibitem{cordes-durham}
M.~Cordes and M.~G. Durham, ``Boundary convex cocompactness and stability of
  subgroups of finitely generated groups,''
  \href{http://dx.doi.org/10.1093/imrn/rnx166}{{\em International Mathematics
  Research Notices} (2017) rnx166}.
  \url{http://dx.doi.org/10.1093/imrn/rnx166}.

\bibitem{CaMa}
C.~H. {Cashen} and J.~M. {Mackay}, ``{A Metrizable Topology on the Contracting
  Boundary of a Group},'' {\em ArXiv e-prints} (Mar., 2017) ,
  \href{http://arxiv.org/abs/1703.01482}{{\ttfamily arXiv:1703.01482
  [math.MG]}}.

\bibitem{Antolin:2016aa}
Y.~{Antol{\'{\i}}n}, M.~{Mj}, A.~{Sisto}, and S.~J. {Taylor}, ``{Intersection
  properties of stable subgroups and bounded cohomology},'' {\em ArXiv
  e-prints} (Dec., 2016) , \href{http://arxiv.org/abs/1612.07227}{{\ttfamily
  arXiv:1612.07227 [math.GT]}}.

\bibitem{Tran:2017aa}
H.~C. {Tran}, ``On strongly quasi-convex subgroups,'' {\em ArXiv e-prints}
  (2017) , \href{http://arxiv.org/abs/1707.05581}{{\ttfamily
  arXiv:1707.05581}}.

\bibitem{Mu}
D.~{Murray}, ``{Topology and Dynamics of the Contracting Boundary of Cocompact
  CAT(0) Spaces},'' {\em ArXiv e-prints} (Sept., 2015) ,
  \href{http://arxiv.org/abs/1509.09314}{{\ttfamily arXiv:1509.09314
  [math.GT]}}.

\bibitem{Co17}
M.~{Cordes}, ``{A survey on Morse boundaries and stability},'' {\em ArXiv
  e-prints} (Apr., 2017) , \href{http://arxiv.org/abs/1704.07598}{{\ttfamily
  arXiv:1704.07598 [math.MG]}}.

\bibitem{BB}
W.~Ballmann and S.~Buyalo, ``Periodic rank one geodesics in {Hadamard}
  spaces,'' in {\em Geometric and Probabilistic Structures in Dynamics},
  vol.~469 of {\em Contemporary Math}.
\newblock 2008.

\bibitem{ChMu}
R.~{Charney} and D.~{Murray}, ``{A rank-one CAT(0) group is determined by its
  Morse boundary},'' {\em ArXiv e-prints} (July, 2017) ,
  \href{http://arxiv.org/abs/1707.07028}{{\ttfamily arXiv:1707.07028
  [math.GT]}}.

\bibitem{MoRu}
J.~Russell and S.~Mousley, ``Hierarchically hyperbolic groups are determined by
  their morse boundary,'' {\em preprint} .

\bibitem{BF}
M.~Bestvina and K.~Fujiwara, ``A characterization of higher rank symmetric
  spaces via bounded cohomology,'' {\em Geometric and Functional Analysis}
  {\bfseries 19} (2009) 11--40.

\bibitem{Su}
H.~Sultan, ``Hyperbolic quasi-geodesics in {$\rm{CAT}(0)$} spaces,''
  \href{http://dx.doi.org/10.1007/s10711-013-9851-4}{{\em Geometriae Dedicata}
  {\bfseries 169} no.~1, (Apr, 2014) 209--224}.
  \url{http://dx.doi.org/10.1007/s10711-013-9851-4}.

\bibitem{Kent-Leininger}
A.~Kent and C.~J. Leininger, ``Shadows of mapping class groups: capturing
  convex cocompactness,''
  \href{http://dx.doi.org/10.1007/s00039-008-0680-9}{{\em Geom. Funct. Anal.}
  {\bfseries 18} no.~4, (2008) 1270--1325}.
  \url{http://dx.doi.org/10.1007/s00039-008-0680-9}.

\bibitem{BS}
S.~Buyalo and V.~Schroeder, \href{http://dx.doi.org/10.4171/036}{{\em Elements
  of asymptotic geometry}}.
\newblock EMS Monographs in Mathematics. European Mathematical Society (EMS),
  Z\"urich, 2007.
\newblock \url{http://dx.doi.org/10.4171/036}.

\bibitem{Bowditch91}
B.~H. Bowditch, ``Notes on {G}romov's hyperbolicity criterion for path-metric
  spaces,'' in {\em Group theory from a geometrical viewpoint ({T}rieste,
  1990)}, pp.~64--167.
\newblock World Sci. Publ., River Edge, NJ, 1991.

\end{thebibliography}\endgroup

\end{document}